\documentclass[10pt]{amsart}
\usepackage[utf8]{inputenc}

\usepackage{amssymb,amsmath,amsthm}
\usepackage{mathdots}
\usepackage{xcolor}
\usepackage{graphicx}
\usepackage{url}

\newtheorem{thm}{Theorem}[section]

\newtheorem{prop}[thm]{Proposition}

\theoremstyle{definition}
\newtheorem{rem}[thm]{Remark}
\newtheorem{ques}[thm]{Question}

\newtheorem{conj}[thm]{Conjecture}
\newtheorem{ex}[thm]{Example}

\newcommand{\R}{\mathbb{R}}
\newcommand{\C}{\mathbb{C}}
\newcommand{\N}{\mathbb{N}}
\newcommand{\Z}{\mathbb{Z}}
\newcommand{\Q}{\mathbb{Q}}

\newcommand{\cal}{\mathcal}

\begin{document}

\title{Binary sequences meet the Fibonacci sequence}

\thanks{\textsuperscript{*} Corresponding author \\
E-mail: \url{piotr.miska@uj.edu.pl}, \url{miskap.ujs.sk} (P. Miska),  
\url{bartosz.sobolewski@uj.edu.pl} (B. Sobolewski), \url{maciej.ulas@uj.edu.pl} (M. Ulas)}

\maketitle

    \begin{center}{ \large 
Piotr Miska \textsuperscript{a, b}, Bartosz Sobolewski \textsuperscript{a, *}, Maciej Ulas \textsuperscript{a} }
\end{center}

\vspace{1em}

{\it \small
\textsuperscript{a } Jagiellonian University, Faculty of Mathematics and Computer Science, Institute of
Mathematics, {\L}ojasiewicza 6, 30-348 Krak\'ow, Poland 

\textsuperscript{b } J. Selye University, Faculty of Economics and Informatics, Department of Mathematics
Hradn\'a 167/21, P. O. Box 54, 945 01 Kom\'arno, Slovakia}

\begin{abstract}
We introduce a new family of meta-Fibonacci sequences $(f(n))_{n\in\mathbb{N}}$, governed by the recurrence relation $$f(n)=af(n-u_{n}-1)+bf(n-u_{n}-2),$$ where $\mathbf{u}=(u_{n})_{n\in \mathbb{N}}$ is a sequence with values $0,1$. Our study focuses on the properties of the sequence of quotients $h(n) = f(n+1)/f(n)$ and its set of values $\mathcal{V}(f)=\{h(n): n \in \mathbb{N}\}$ for various $\mathbf{u}$. We give a sufficient condition for finiteness of $\mathcal{V}(f)$ and automaticity of $(h(n))_{n \in \mathbb{N}}$, which holds in particular when $\mathbf{u}$ is the famous Prouhet-Thue-Morse sequence. In the automatic case, a constructive approach is used, with the help of the software \texttt{Walnut}. On the other hand, we prove that the set $\cal{V}(f)$ is infinite for other special binary sequences $\mathbf{u}$, and obtain a trichotomy in its topological type when $\mathbf{u}$ is eventually periodic.
\end{abstract}

\section{Introduction and motivation}\label{sec:intro}

Let $a, b\in\Z$ and recall that by a binary linear recurrence sequence we mean a sequence $(g(n))_{n\in\N}$ satisfying the recurrence relation 
\begin{equation}\label{binlin}
   g(n)=ag(n-1)+bg(n-2), 
\end{equation}
where $g(0), g(1)$ are given. The most extensively studied example is of course the Fibonacci sequence, which arises for $a=b=1$ and $g(0)=0,  g(1)=1$. In the literature there are many variations of binary recurrence sequences, including so- called meta-Fibonacci sequences. What is a meta-Fibonacci sequence? In the most general terms, it is a solution of a recurrence relation of the form
\begin{equation}\label{genmeta}
f(n)=f(r_{1}(n))+f(r_{2}(n)),
\end{equation}
where $r_{1}(n), r_{2}(n)$ are certain expressions involving $n, f(n-1), \ldots, f(n-k)$ and possibly other functions of $n$. We note that the class of sequences defined in this way is much broader than the one introduced by Conolly in \cite{Con}. Special meta-Fibonacci sequences were investigated in \cite{Gol, Hof, Rus}. Probably the most famous one is Hofstadter's $Q$-sequence (see \cite[page 137]{Hof}), defined by $Q(1)=Q(2)=1$ and 
$$
Q(n)=Q(n-Q(n-1))+Q(n-Q(n-2))
$$
for $n\geq 2$. It is an open question whether the sequence is infinite, i.e., whether for each $n\in\N$ one can compute the value of $Q(n)$. 

In the paper, we consider a new variation of binary recurrence sequences, where the indices on the right-hand side of \eqref{binlin} are additionally shifted, depending on a given sequence $\mathbf{u}=(u_n)_{n \in \N}$ with values in the set $\{0, 1\}$. More precisely, we are interested in $(f(n))_{n \in \N}$ satisfying
\begin{equation}\label{fdefinition}
f(n)=af(n-u_{n}-1)+bf(n-u_{n}-2),
\end{equation}
where $a, b\in\Z$. In other words, the values $u_n$ are responsible for switching between a binary and ternary recurrence relation. For the most part, we consider the initial conditions
$$f(0)=1, \quad f(1)=1, \quad f(2)= a+b.$$
This will turn out to be a natural choice in the context of our results --- see Section \ref{sec:notation} for a short discussion.

 The present study arose from the observation that in the case $\mathbf{u}=(T_{n})_{n\in\N}$, the famous Prouhet--Thue--Morse sequence (PTM sequence for short), subsequent terms $f(n)$ have surprisingly many common prime factors, regardless of the parameters $a,b$. Contrast this with usual binary recurrence sequences \eqref{binlin}, where for $a,b,g(0),g(1)$ pairwise coprime one can show inductively that $g(n), g(n+1)$ are coprime for all $n \in \N$.
 This leads us to studying the set of ratios
 $$\cal{V}(f) := \{f(n+1)/f(n): n\in \N\},$$
 which turns out to be finite in the case $\mathbf{u}=(T_{n})_{n\in\N}$, as we prove in Theorem \ref{PTMcase} below. We note that for a binary linear recurrence sequence $(g(n))_{n\in\N}$ the sequence of ratios of consecutive terms is known to be either convergent, periodic, or dense in some circle or line in the complex plane (see \cite{BHP, BK}). It is thus interesting to see what may change in the structure of the set (or sequence) of ratios if we modify the recurrence relation as described. 

Apart from the PTM sequence, we study the more general case when $\mathbf{u}$ belongs to the class of \textit{automatic sequences}, which supplies many natural examples of $0$-$1$ sequences. We recall the definition of automatic sequences and their essential properties in Section \ref{sec:PTM}. A different kind of connection between meta-Fibonacci sequences (more specifically, a relative of Hofstadter's $Q$-sequence) and automatic sequences is investigated in \cite{AS13}.

 Finiteness of $\cal{V}(f)$ turns out to be a direct consequence of the key property that consecutive terms $0,1,0$ occur in the PTM sequence with bounded gaps.  We thus also investigate the case when the sequence $\mathbf{u}$ does not satisfy said property but potentially has some other structure.
 
To conclude this section, we describe the content of the paper in some more detail. In Section \ref{sec:PTM} we investigate the solutions of \eqref{fdefinition} in the case when $\mathbf{u}=(T_{n})_{n \in \N}$ is the PTM sequence. We prove that the set $\cal{V}(f)$ of quotients of consecutive terms of $(f(n))_{n\in\N}$ contains exactly $7$ elements (Theorem \ref{PTMcase}). Motivated by this, in Section \ref{sec:fin_auto} we obtain a similar result for a broader class of sequences, where the recurrence \eqref{genmeta} is governed by two binary sequences. Section \ref{sec:construction} describes another, more constructive approach to proving these results using the software \texttt{Walnut} \cite{Mou}. 
In Section \ref{sec:general} we investigate the case when $\textbf{u}$ contains arbitrarily long blocks of zeros, and prove that, under certain conditions, the set $\cal{V}(f)$ contains the set of ratios $\cal{V}(g)$ of a binary recurrence sequence $(g(n))_{n \in \N}$. We also consider ultimately periodic $\textbf{u}$, in which case the set $\cal{V}(f)$ turns out to be finite, have finitely many accumulation points, or be dense in $\R$. Finally, in Section \ref{sec:comments} we present further remarks concerning possible generalization and offer some general conjectures.

\section{Notation and basic observations} \label{sec:notation}
Throughout the remainder of the paper we assume that $\mathbf{u}$ is a given binary sequence and let $(f(n))_{n \in \N}$ denote the sequence defined by the recurrence \eqref{fdefinition} and initial conditions $f(0)=f(1)=1, f(2)=a+b$. We consider the sequence of ratios
$$ h(n) := \frac{f(n+1)}{f(n)}, $$
which will be our main object of study along with the set $\cal{V}(f) = \{h(n): n \in \N  \}$.
We may view $f(n)$ and $h(n)$ in two ways:
\begin{enumerate}
    \item[(i)] as number sequences for fixed $a,b \in \Z$;
    \item[(ii)] as sequences of rational functions in variables $a,b$.
\end{enumerate}
Usually, both cases can be treated in the same way so by default when writing $f(n), h(n)$ we mean the first interpretation.
The rational function approach is convenient when we need to eliminate the possibility $f(n)=0$, in which case $h(n)$ would be undefined. If it is not clear from the context, we will specify when this latter interpretation is to be applied.

We now explore some basic properties of $h(n)$. First, we have $h(0)=1, h(1) = a+b$, and the following recurrence relation:
\begin{equation} \label{h_relation}
     h(n) = \begin{cases}
a + \frac{b}{h(n-1)} &\text{if } u_{n+1}=0, \\
\frac{a}{h(n-1)}+\frac{b}{h(n-1) h(n-2)} &\text{if } u_{n+1}=1,
\end{cases}
\end{equation}
for $n \geq 2$.

Moreover, observe that if $u_n = 0, u_{n+1} = 1$, then $f(n) = f(n+1)$ so $h(n) = 1$. Further still, if additionally $u_{n+2} = 0$, then $f(n+2) = (a+b) f(n+1)$, which yields $h(n+1) = a+b$. In other words, an occurrence of consecutive terms $0,1,0$ in the sequence $\mathbf{u}$ ``resets'' the sequence $(h(n))_{n \in \N}$. This remains true even if we set arbitrary initial values $f(0),f(1),f(2)$. This explains the choice $f(0)=f(1)=1, f(2)=a+b$, which essentially makes $n=0$ the first ``reset point'' and allows for nicer statements of results. Hence, it is useful to consider the set
$$ \mathcal{R}(\mathbf{u}) := \{n \in \N: u_n=0, u_{n+1}=1, u_{n+2} = 0 \} \cup \{0\} = \{n_0 = 0 < n_1 < \cdots   \}.  $$
As we shall see, certain properties of $(h(n))_{n \in \N}$ are tied to whether $\mathcal{R}(\mathbf{u})$ is infinite and there exists $c \geq 1$ such that $n_{k+1}-n_k \leq c$.  
If both conditions hold, we will say that $\mathcal{R}(\mathbf{u})$ has bounded gaps, or, when the value of $c$ is relevant --- has gaps bounded by $c$.
 When $\mathbf{u}$ does not contain $0,1,0$ as a subsequence, all the terms $h(n)$ may depend on the choice of initial values $f(0),f(1),f(2)$. Nevertheless, in such case our results still stand, up to minor modifications.

Other patterns of note in $\mathbf{u}$ include long sequences of $0$s and $1$s. Indeed, as long as $u_n = \varepsilon$ is constant, the values $f(n)$ behave like a binary or ternary recurrence sequence, depending on whether $\varepsilon=0$ or $\varepsilon=1$. In such a case, standard results concerning recurrence sequences may be applicable.

Finally, if $\mathbf{u}$ contains for some $k \geq 2$ the subsequence $(u_k,u_{k+1},u_{k+2})=(0,0,0)$ then substituting it with $(0,1,1,0)$ is equivalent to duplicating the term $f(k)$ (with other terms unchanged). Indeed, if $(f'(n))_{n \in \N}$ is the sequence obtained by this operation, then  
\begin{align*}
   f'(k)=&af'(k-1)+bf'(k-2)=af(k-1)+bf(k-2)=f(k),\\
    f'(k+1)=&af'(k-1)+bf'(k-2)=af(k-1)+bf(k-2)=f(k),\\
    f'(k+2)=&af'(k)+bf'(k-1)=af(k)+bf(k-1)=f(k+1),\\
    f'(k+3)=&af'(k+2)+bf'(k+1)=af(k+1)+bf(k)=f(k+2),
\end{align*}
and inductively $f'(n+1)=f(n)$ for $n \geq k+2$.
Conversely, the substitution $(0,1,1,0) \mapsto (0,0,0)$ deletes such a duplicate. In either case, the set $\mathcal{V}(f)$ stays the same. 

\section{The case of the PTM sequence}\label{sec:PTM}

As a starting point of our investigation, we consider the case where $\mathbf{u}$ is the PTM sequence $(T_{n})_{n\in\N}$. Recall that
$$
T_{n}=s_{2}(n)\bmod{2},
$$
$s_{2}(n)$ is the sum of binary digits of $n$. We have $T_{0}=0, T_{1}=1$ and for $n\in\N_{+}$ the recurrence relations
$$T_{2n}=T_{n}, \qquad T_{2n+1}=1-T_{n}.$$

This is a $2$-automatic sequence, namely there exists a deterministic finite automaton with output (DFAO), shown in Figure \ref{fig:PTM}, which reads the binary representation of $n$ digit by digit and outputs $T_n$.  The expression ``$x/y$'' in a node means that $x$ is its index and $y$ --- its output.

\begin{figure}[h!]
    \centering
    \includegraphics[width=0.5\linewidth]{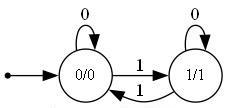}
    \caption{A DFAO generating the PTM sequence}
    \label{fig:PTM}
\end{figure}

More generally, sequences obtained in this fashion using the base-$k$ representation of $n$ are called $k$-automatic. The class of obtained sequences is identical, regardless of whether the automata read the most significant digit first (msd-first for short) or the least significant one (lsd-first). Equivalently, a sequence $(a_n)_{n \in \N}$ is $k$-automatic if its \emph{$k$-kernel}
$$
K_k(\mathbf{a}) = \{(a_{k^jn+i})_{n\in\N}: j \in \N, 0 \leq i < k^j \},
$$
is a finite set. A panorama of properties and applications of automatic sequences is presented in the monograph of Allouche and Shallit \cite{AS03a}. When dealing with automatic sequences, we are going to use the free software \texttt{Walnut} \cite{Mou}, which implements a decision procedure for proving their properties. 
All relevant files are available in the GitHub repository \cite{Git}: \\
 \url{https://github.com/BartoszSobolewski/Binary-sequences-meet-Fibonacci} \\
 An introduction to \texttt{Walnut} and a large collection of applications can be found in the book of Shallit \cite{Sha}. 

To formulate the main result of this section we introduce auxiliary polynomials $d_{i}=d_{i}(a,b)$, given by
\begin{align*}
d_{-1}&=d_{0}=1,\\
d_{1}&=a+b,\\
d_{2}&=a^2+(a+1)b, \\
d_{3}&=a^3+b(a^2+2a+b),             \\
d_{4}&=a^4+b(a^3+3a^2+2ab+b),              \\   d_{5}&=a^5+b(a^4+4a^3+3a^2b+3ab+b^2),      \\
d_{6}&=a^6+b(a^5+5a^4+4a^3b+6a^2b+3ab^2+b^2). 
\end{align*}

\begin{thm}\label{PTMcase}
If $\mathbf{u}$ is the PTM sequence, then $(h(n))_{n \in \N}$ is $2$-automatic and takes values in the set
$$
\cal{V}(f)=\left\{h_{0}, h_{1}, h_{2}, h_{3}, h_{4}, h_{5}, h_{6}\right\},
$$
where
$$
h_{i}=\frac{d_{i}}{d_{i-1}}.
$$
Moreover, for each $i\in\{0, 1, \ldots, 6\}$ there are infinitely many values of $n$ such that $f(n+1)=h_{i}f(n)$ (here $h_{i}$ is treated as an element of $\Q(a,b)$).
\end{thm}
\begin{proof}
By generating many initial terms $h(n)$ and using the method described in \cite[Section 5.6]{Sha}, we can guess a $23$-state lsd-first DFAO $\mathcal{K}_0$ computing the sequence $(k_n)_{n \in \N}$ such that $h(n) = h_{k_n}$. 

We now use \texttt{Walnut} to verify that this DFAO is indeed correct.
The DFAO $\mathcal{K}_0$ is stored in \texttt{Walnut}-compatible format in the file \texttt{K0.txt}, which can be found in the repository \cite{Git} along with all other automata used in this paper. In order to work simultaneously with the built-in DFAO for the PTM sequence (which uses msd-first convention), we execute the command

\texttt{reverse K K0:}

It creates a msd-first DFAO $\mathcal{K}$, stored in the file \texttt{K.txt}, generating the same sequence $(k_n)_{n \in \N}$. It is shown in Figure \ref{fig_K}. We have $h(n) = h_{k_n}$ for $n=0,1$ and need to check that the relation \eqref{h_relation} holds when $h(n), h(n-1), h(n-2)$ are replaced with $h_{k_{n}}, h_{k_{n-1}}, h_{k_{n-2}}$. To check that the first case is satisfied (with the index shifted by $2$), we execute the \texttt{Walnut} command

\texttt{def T0 "En K[n]=x \& K[n+1]=y \& T[n+2]=@0":}

It creates a $2$-DFA which reads binary representations of $x,y \in \N$ in parallel and accepts precisely those pairs $(x,y)$ for which there exists $n \in \N$ satisfying:
$$ k_n = x, \quad k_{n+1} = y, \quad T_{n+2} = 0. $$
Inspecting the DFA, we can see that precisely the following pairs $(x,y)$ are accepted:
$$ (0,1), (1,2), (2,3), (3,4), (4,5), (5,6).  $$
For such $(x,y)$ we have $h_y = a + b/h_x$, as desired.

Similarly, to verify the second case, we execute the command:

\texttt{def T1 "En K[n]=x \& K[n+1]=y \& K[n+2]=z \& T[n+3]=@1":}

For each triple $(x,y,z)$ accepted by the resulting $2$-DFA, namely
$$(0,1,0),(1,0,2),(1,2,0),(2,0,3),(2,3,0),(3,4,0),(4,0,5),(4,5,0),(5,6,0),$$
we can check that
$$  h_z = \frac{a}{h_x}  + \frac{b}{h_x h_y}. $$

Finally, through inspection of Figure \ref{fig_K} one can see that for each output $k \in \{0,\ldots,6\}$ there exist infinitely many $n \in \N$ such that $k_n = k$.
\end{proof}

\begin{figure}[h] 
\caption{A $2$-DFAO generating the sequence $(k_n)_{n \in \N}$} \label{fig_K}
\centering
\includegraphics[width=\textwidth]{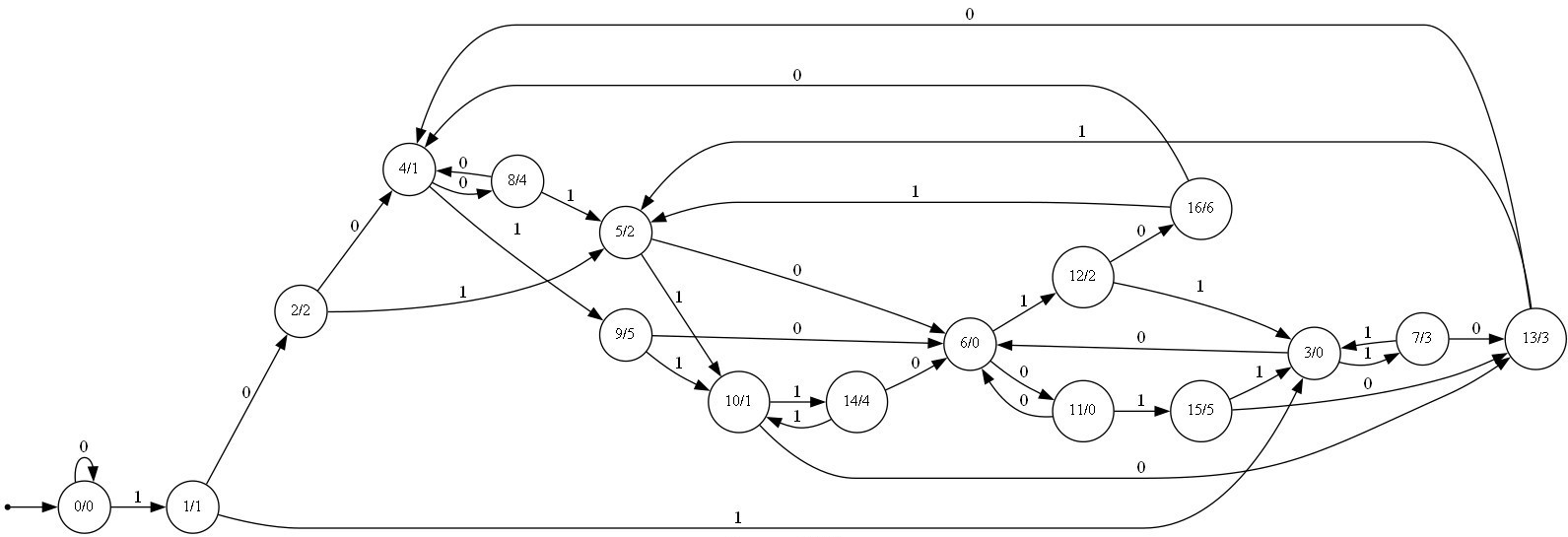}
\end{figure}

Now, the PTM sequence is a very special case due to the property that it contains at most two consecutive $1$'s. Here we sketch another way to determine $\cal{V}(f)$ using this fact. Consider substituting each block $(0,1,1,0)$ in $(T_n)_{n\in\N}$ with $(0,0,0)$ and let $(T'_n)_{n\in\N}$ denote the resulting binary sequence. If $(f'(n))_{n \in \N}$ is the associated sequence of the form \eqref{fdefinition}, then by the observation at the end of Section \ref{sec:notation} we get $\cal{V}(f) = \cal{V}(f')$. But   $(T'_n)_{n\in\N}$ consists of strings of $0$s separated by single $1$s, each marking a ``reset point'' for the sequence of quotients $f'(n+1)/f'(n)$. Hence, these quotients coincide with initial values of $g(n+1)/g(n)$, where $g$ is the binary recurrence sequence given by $g(n+2)=ag(n+1)+bg(n), g(0)=1,g(1)=1$. One can show that $1$s appear in $(T'_n)_{n\in\N}$ at distance at most $7$ (and this distance is attained), which means that $\cal{V}(f) = \{g(n+1)/g(n): 0 \leq n < 7   \}$.

Concerning automaticity, in Sections \ref{sec:fin_auto} and \ref{sec:construction} we give constructive versions of the proof in a more general setting.

\begin{rem}
If we consider other, more complicated automatic sequences in place of $T_n$, the size of the DFAO generating $(h(n))_{n \in \N}$ may change drastically. For example, replacing $T_n$ with a $0$-$1$ variant of the Rudin--Shapiro sequence seems to produce an lsd-first DFAO with $214$ states and $38$ possible outputs $h(n)$.
\end{rem}

\begin{rem}
If we treat $a, b$ as fixed integers (not both equal to 0), then not all quotients $h_{i}$ given in the statement of Theorem \ref{PTMcase} are defined. More precisely, let $C_{i}$ be the curve defined by the equation $d_{i}(a, b)=0$. We see that for each integer point lying on $C_{i}$ we have that $h_{i+1}$ is not defined. If $i=1$ and $(a, b)$ is integer point on $C_{1}$, i.e., $b=-a$ the set of values of the sequence $(f(n))_{n\in\N}$ is infinite as well as the set of $n\in\N$ such that $f(n)=0$.

On the other hand, for $i=2, 3, 4, 5, 6$ the set $C_{i}(\Z)$ of all integral points is finite. To show this, we perform a case-by-case analysis.
First of all let us note that if $a=0$ (or $b=0$) then in each case we get  $b=0$ ($a=0$). Then $f(n)=0$ for $n\geq 2$. We thus assume that $ab\neq 0$.

If $i=2$ then $a^2 + b + a b=0$ and thus $b=-a^2/(a+1)$. We get that $a=0$ or $a=-2$. If $a=-2$ then $b=4$. In this case we have $f(n)=0$ for $n\geq 3$.

Let $i=3$. We consider the equation $d_{3}(a, b)=a^3+a^2b+2ab+b^2=0$ defining the curve $C_{3}$. Let us write $a=ua_{1}, b=ub_{1}$, where $u=\gcd(a, b)$ and $\gcd(a_{1}, b_{1})=1$. After dividing $d_{3}(ua_{1}, ub_{1})$ by $u^2$ we are left with the equation $a_1(a_1^2u+a_1 b_1u+2b_1)=-b_1^2$. Thus, each prime factor of $a_{1}$ is a divisor of $b_{1}$. However, $\gcd(a_{1}, b_{1})=1$ and in consequence $a_{1}\in\{-1, 1\}$, i.e., $a|b$. If $a_{1}=1$ then solving resulting (linear) equation with respect to $u$, we compute
$$
u=-b_{1}-1+\frac{1}{b_{1}+1}.
$$
Because $b_{1}\in\Z$ we get that $b_{1}=-2$ or $b_{1}=0$. In both cases we get that $u=0$ - a contradiction. Exactly the same reasoning works in the case when $a_{1}=-1$.

If $i=4$ then again we deduce that $a|b$ and writing $b=ua, u\neq 0$, we left with the equation $(2a+1)u^2+a(a+3)u+a^2=0$. The discriminant with respect to $u$ is equal to $a^2(a^2-2a+5)$. Thus, the expression $a^2-2a+5$ needs to be a square of integer. The only possibility is $a=1$. Then $u=-1$,  and thus $b=-1$. In this case the sequence $(f(n))_{n \in \N}$ itself is $2$-automatic and takes values $0,1,-1$, all infinitely often. To see this, one can repeat the argument in the proof of Theorem \ref{PTMcase}: guess a DFAO $\mathcal{F}$ for $(f(N))_{n \in \N}$, and then check that the sequence generated by $\mathcal{F}$ indeed satisfies the recurrence \eqref{fdefinition}.
We leave out the details; the interested reader may find relevant \texttt{Walnut} files and commands in the repository \cite{Git}.

If $i=5$ then again we deduce that $a|b$ and writing $b=ua, u\neq 0$, we left with the equation $(u+1)a^2+u(3u+4)a+(u+3)u^2=0$. If $u=-1$ we get that $a=2$ and then $b=-2$. We thus assume that $u\neq -1$. Solving this equation with respect to $a$ we get that
$$
a=-\frac{u(3u+4\pm \sqrt{5u^2+8u+4})}{2(u+1)}.
$$
Because $a, u$ need to be integers, we get that $u+1|1+\sqrt{5u^2+8u+4}$ or $u+1|\sqrt{5u^2+8u+4}-1$. First, consider the case when $u+1|1+\sqrt{5u^2+8u+4}$. Because $0<1+\sqrt{5u^2+8u+4}<3|u+1|$ we see that for some $v\in\{-2, -1, 1, 2\}$ we have the equality $(v(u+1)-1)^2=5u^2+8u+4$. If $v=1$ we get that $u=-1$ - the case we already considered. If $v=2$, then $u=-1$ or $u=-3$. In case of $u=-3$ we get that $a=0$ or $a=15/2$ - a contradiction. If $v=-1$, then $u=-1$ or $u=0$ - the cases we already considered. Finally, if $v=-2$, then $u=-1$ (the case we already considered) or $u=5$. The case $u=5$ implies that $a=-40/3$ or $a=-5/2$. Now, we examine the case when $u+1|\sqrt{5u^2+8u+4}-1$. Because $\sqrt{5u^2+8u+4}-1<3|u+1|$ we see that for some $v\in\{-2, -1, 0, 1, 2\}$ we have the equality $(v(u+1)+1)^2=5u^2+8u+4$. One can check that we get no new solutions $(a,b)$ in this case. Thus, the only integer solution of $d_{5}(a, b)=0$ satisfying the condition $ab\neq 0$ is $(a, b)=(2, -2)$. In this case we have $f(n)=0$ for $n\geq 10$.

The case $i=6$ is more complicated. The polynomial $d_{6}(a,b)$ can be written as $d_{6}(a,b)=a^5(a+b)+l.o.t.$ and thus it satisfies so called the Runge condition (see for example \cite{Walsh}). This allows us to apply Runge's method and deduce that the only solution is $a=b=0$. We omit the details.  
\bigskip

As a final remark, note that each curve $C_{i}$ is rational over $\Q$. Indeed, to get a parametrization of $C_{i}$ it is enough to take $b=ta^2$, where $t$ is rational parameter, and solve the corresponding equation with respect to $a$.
\end{rem}

\section{Finiteness and automaticity} \label{sec:fin_auto}

In this section we show that Theorem \ref{PTMcase} is not just an isolated property of the PTM sequence sequence and offer more general results in this direction. Consider a generalization of the recurrence relation \eqref{fdefinition}, where two independent binary sequences ${\bf u}=(u_n)_{n\in\N}, {\bf v}=(v_n)_{n\in\N} \in\{0,1\}^\N$ appear in the index. More precisely, we let
\begin{equation} \label{fdefinition2}
    \tilde{f}(n) = a\tilde{f}(n-1-u_n) + \tilde{f}(n-2-v_n),
\end{equation} 
where again $a, b \in\Z$ and $\tilde{f}(0)=\tilde{f}(1) = 1, \tilde{f}(2)=a+b$. Also define the associated sequence of quotients by $\tilde{h}(n) = \tilde{f}(n+1)/\tilde{f}(n)$. In particular, if $\mathbf{u} = \mathbf{v}$, then we recover the original definition of $f(n)$ and $h(n)$. The following result gives sufficient conditions for finiteness of $\mathcal{V}(\tilde{f})$ and automaticity of $(\tilde{h}(n))_{n \in \N}$.

\begin{thm}\label{main}
Let $\tilde{f}(n)$ be defined by \eqref{fdefinition2}. Consider the set
\begin{align*}
\tilde{\mathcal{R}}(\mathbf{u}) &=\{n\in\N:\;(u_n, u_{n+1} ) \in \{ (0,1), (1,0)\}\;\mbox{and}\;(v_n, v_{n+1}, v_{n+2})= (0,1,0)\} \cup \{0\},\\
       &=\{n_0 = 0 < n_{1} < \cdots\}.
\end{align*}
If $\tilde{\mathcal{R}}(\mathbf{u})$ has gaps bounded by $c$, then
$\# \cal{V}(\tilde{f})\leq 4^{c-1}+1$.

Moreover, if the sequences ${\bf u}, {\bf v}$ are additionally $k$-automatic, then $(\tilde{h}(n))_{n \in \N}$ is also $k$-automatic.
\end{thm}
\begin{proof}
It is easy to verify that $\tilde{h}(0)=1, \tilde{h}(1)=a+b$, and for $n \geq 2$ we have
\begin{equation}\label{cases}
  \tilde{h}(n) = \begin{cases}
a + \dfrac{b}{\tilde{h}(n-1)} &\text{if } u_{n+1} = v_{n+1}=0, \\
a + \dfrac{b}{\tilde{h}(n-1)\tilde{h}(n-2)} &\text{if } u_{n+1} =0, v_{n+1}=1, \\
\dfrac{a+b}{\tilde{h}(n-1)} &\text{if } u_{n+1} =1, v_{n+1}=0, \\
\dfrac{a}{\tilde{h}(n-1)}+\dfrac{b}{\tilde{h}(n-1) \tilde{h}(n-2)} &\text{if } u_{n+1} = v_{n+1}=1.
\end{cases}
\end{equation}
As was previously the case for the set $\mathcal{R}(\mathbf{u})$, the indices $n=n_i$ act as ``reset points'' for the terms $\tilde{h}(n)$ in the sense that $\tilde{h}(n_i), \tilde{h}(n_i+1)$ are independent of earlier values.
Indeed, by a simple calculation, when $(v_n, v_{n+1}, v_{n+2}) = (0,1,0)$, we get $$ (\tilde{h}(n), \tilde{h}(n+1)) = \begin{cases}
   (1,a+b) &\text{if } (u_n,u_{n+1}) =(0,1), \\
   \left(\frac{a^2+a b+b}{a+b},\frac{a^3+a^2 b+2 a b+b^2}{a^2+a b+b}\right) &\text{if } (u_n,u_{n+1},u_{n+2}) =(1,0,0), \\
   \left(\frac{a^2+a b+b}{a+b},\frac{(a+b)^2}{a^2+a b+b}\right) &\text{if } (u_n,u_{n+1},u_{n+2}) =(1,0,1). \\
\end{cases}$$
Also note that $n_0=0$ falls under the first case, regardless of the initial terms of $\mathbf{u}, \mathbf{v}$.

Hence, for each $i$ there are $2$ possible values of $\tilde{h}(n_i)$ and $3$ possible values of $\tilde{h}(n_i+1)$. Now, consider $\tilde{h}(n_i+j)$, where let $2 \leq j \leq n_{i+1}-n_i-1 \leq c-1$. By repeatedly applying the recurrence relation \eqref{cases} we can see that $\tilde{h}(n_i +j)$ depends only on the pair $(\tilde{h}(n_i), \tilde{h}(n_i+1))$, taking at most distinct $3$ values, and binary sequences $(u_{n_i+3}, \ldots, u_{n_i+j+1})$, $(v_{n_i+3}, \ldots, v_{n_i+j+1})$ of length $j-1$. Therefore, the number of possible values of $\tilde{h}(n_i+j)$ is at most $3 \cdot 4^{j-1}$.
As a consequence, $\tilde{h}(n)$ may attain at most
$$ 2 + 3 + \sum_{j=2}^{c-1} 3 \cdot 4^{j-1} = 4^{c-1}+1$$
distinct values.

Now, assume that $(u_n)_{n \in \N}$ and $(v_n)_{n \in \N}$ are $k$-automatic sequences. Define $U_n = (u_n,\ldots,u_{n+c})$ and $V_n = (v_n,\ldots,v_{n+c})$. For any $n \in \N$ there exists some $i$ such that $n \leq n_i \leq n+c-1$ so by the earlier discussion $\tilde{h}(n+c)$ depends only on $U_n$ and $V_n$. But the sequences $(U_n)_{n \in \N}$, $(V_n)_{n \in \N}$ are $k$-automatic, and thus so is $(\tilde{h}(n+c))_{n \in \N}$. Since shifting a sequence backwards and modifying finitely many terms does not affect $k$-automaticity either, we deduce that $(\tilde{h}(n))_{n \in \N}$ is $k$-automatic as well.
\end{proof}

Taking $\mathbf{u} = \mathbf{v}$ we obtain a more particular result concerning the original sequence $(f(n))_{n \in \N}$. One can then improve the upper bound on $\#\mathcal{V}(f)$ by a similar reasoning as in the proof above. 

\begin{prop} \label{main2}
    Let $f(n)$ be defined by \eqref{fdefinition} and $\mathcal{R}(\mathbf{u}) =\{n_0 = 0 < n_{1} < \cdots\}$.
If $\mathcal{R}(\mathbf{u})$ has gaps bounded by $c$, then $\# \mathcal{V}(f)\leq 2^{c-1}$.

Moreover, if the sequence $\mathbf{u}$ is additionally $k$-automatic, then $(h(n))_{n \in \N}$ is also $k$-automatic.
\end{prop}

\begin{rem}
 Theorem \ref{main} and Proposition 4.2 can be extended to general automatic sequences, associated with numeration systems other than the standard base-$k$ expansions. Some of these cases, such as Fibonacci-automatic sequences, can be handled by \texttt{Walnut}. 
 This also applies to later results dealing with automaticity, including Proposition \ref{u_automatic} and Theorem \ref{higher_order}. 
\end{rem}

A question arises whether given sufficient conditions for finiteness and automaticity are also necessary. For fixed $a,b \in \Z$ this is not the case in general. As we show in Section \ref{sec:general}, the sequence $(h(n))_{n \in \N}$ may be eventually periodic (and thus $k$-automatic for all $k \geq 2$) even though the set $\mathcal{R}(\mathbf{u})$ does not have bounded gaps.

The situation appears to be different when the $f(n)$ are treated as polynomials in $a,b$. Assume first that $\mathcal{V}(f)$ is finite. If $\mathcal{R}(\mathbf{u})$ does not have bounded gaps, then by the pigeonhole principle there exist indices $n < n'$ satisfying the following conditions:
\begin{itemize}
\item $(h(n),h(n+1)) = (h(n'),h(n'+1))$,
    \item there is no $n_i \in \mathcal{R}(\mathbf{u})$ such that $n < n_i < n'$.
\end{itemize}
In particular, we would obtain a new way to ``reset'' the sequence $(h(n))_{n \in \N}$, which would involve an implausible amount of cancellation when applying the recurrence relation \eqref{h_relation}.
We have verified that for all possible initial segments $(u_n)_{0 \leq n \leq 20}$ starting with $0,1,0$ and not containing any other instance of this block, the equality $(h(n),h(n+1)) = (h(n'),h(n'+1))$ never occurs. Hence, it seems likely that $\mathcal{R}(\mathbf{u})$ having bounded gaps is necessary for finiteness of $\mathcal{V}(f)$. Similarly, we expect the same for the sets $\tilde{\mathcal{R}}(\mathbf{u})$ and $\mathcal{V}(\tilde{f})$ appearing in Theorem \ref{main}. Unfortunately, we have not been able to prove these statements.

On the other hand, in the case of automaticity it is rather simple to show that $(h(n))_{n \in \N}$ being automatic (which entails finiteness of $\mathcal{V}(f)$) implies the same for $\mathbf{u}$.

\begin{prop} \label{u_automatic}
    If the sequence $(h(n))_{n \in \N}$ (treated as elements of $\Q(a,b)$) is $k$-automatic, then $\mathbf{u}$ is also $k$-automatic.
\end{prop}
\begin{proof}
    We are going to use relations \eqref{h_relation} to show that for each $n \geq 2$ the triple $(h(n), h(n-1), h(n-2))$ uniquely determines $u_{n+1}$. Suppose this is not the case, namely
$$ a+ \frac{b}{h(n-1)} = h(n) = \frac{a}{h(n-1)} + \frac{b}{h(n-1)h(n-2)}. $$
Letting $a=b=1$, we get $h(n-1) = 1/h(n-2)$, which contradicts the relations \eqref{h_relation}.

It remains to note that $(h(n+2), h(n+1), h(n))_{n \geq 0}$ is also $k$-automatic, thus so is $(u_{n+3})_{n \in \N}$.
\end{proof}

However, it turns out that the same property is not true for the more general recurrence relation \eqref{fdefinition2}, as the following example shows.

\begin{ex}
   Let $(s_n)_{n \in \N}$ be any binary sequence, and put
   $$  
   u_n = \begin{cases}
       0 &\text{if } n \equiv 0,3 \pmod{4}, \\
       1  &\text{if } n \equiv 1 \pmod{4}, \\
       s_n &\text{if } n \equiv 2 \pmod{4}.
   \end{cases}
   \qquad
   v_n = \begin{cases}
       0 &\text{if } n \equiv 0,2,3 \pmod{4}, \\
       1  &\text{if } n \equiv 1 \pmod{4}, 
   \end{cases}
   $$
   Then the corresponding sequence $\tilde{h}(n)$ consists of periodically repeating values:
   $$  1,a+b,\frac{a^2+a b+b}{a+b},\frac{a^3+a^2 b+2 a b+b^2}{a^2+a b+b}, $$
   and thus is $k$-automatic for all $k$.
   But if we pick a non-automatic sequence $(s_n)_{n \in \N}$, for example the characteristic sequence of squares, then $\mathbf{u}$ is non-automatic too.
\end{ex}

\section{Construction of automata using \texttt{Walnut}} \label{sec:construction}

In this section we sketch a way of constructing a DFAO generating the sequence $(\tilde{h}(n))_{n \in \N}$ under the assumptions of Theorem \ref{main}. The ``heavy lifting'' will be performed by \texttt{Walnut}. 
We illustrate our approach for the special case $\mathbf{u} = \mathbf{v}$ (as in Proposition \ref{main2}) and later discuss the modifications that need to be made for general $\mathbf{u}$, $\mathbf{v}$.

We may try to emulate the proof of Theorem \ref{main} by constructing a 2-DFAO generating $U_n= V_n = (u_n,\ldots, u_{n+c})$, however we have not found a ``nice'' way of doing this in \texttt{Walnut}. Instead, we describe a different approach, involving $1$-uniform transduction of $\mathbf{u}$. For an introduction to transducers and their applications, as well as their usage in \texttt{Walnut} see \cite{SZ23}. 

We first define a transducer $\mathcal{T}$, which reads the terms $u_n$ and mimics the recurrence relation \eqref{h_relation}. Its states are labelled with pairs $(X,Y)$ of rational functions of variables $a,b$, where the initial state is $(1,a+b)$. Starting from the initial state, we inductively define further states and transitions between them as follows. The transition from a state $(X,Y)$ at input $\varepsilon \in \{0,1\}$ leads to the state $(X_\varepsilon, Y_{\varepsilon})$, where
\begin{align*}
    (X_0,Y_0) &= \left(Y, a+\frac{b}{Y} \right), \\
    (X_1,Y_1) &= \left(Y, \frac{a}{Y}+\frac{b}{XY} \right).
\end{align*}
The output of this transition is $X$ (regardless of $\varepsilon$). Comparing this with \eqref{h_relation}, we can see that for
$(X,Y) = (h(n),h(n+1))$ and input $\varepsilon = u_{n+3} \in \{0,1\}$, the transition outputs $h(n)$ and leads to the state $(X_\varepsilon, Y_{\varepsilon})= (h(n+1),h(n+2))$. In particular, the sequence of inputs $0,1,0$ always leads to the initial state in $\mathcal{T}$. We have thus constructed an infinite transducer $\mathcal{T}$, which applied to the shifted sequence $\mathbf{\tilde{u}} = (u_{n+3})_{n \in \N}$ outputs $\mathcal{T}(\tilde{\mathbf{u}})=(h(n))_{n \in \N}$.

If we assume that the set $\mathcal{R}(\mathbf{u})$ has gaps bounded by $c$ (as in Proposition \ref{main2}), then only a finite part of $\mathcal{T}$ is actually reachable. 
We can then modify $\mathcal{T}$ to obtain a finite transducer $\mathcal{T}_c$ such that $\mathcal{T}_c(\tilde{\mathbf{u}}) =\mathcal{T}(\tilde{\mathbf{u}}) = (h(n))_{n \in \N}$. More precisely, let $\ell = \ell(X,Y)$ denote the length of the shortest path from the initial state to $(X,Y)$. Then, we keep intact all states with $\ell < c-1$ and transitions from them. The states with $\ell=c-1$ are left unchanged but all transitions from them are directed to the initial state (without modifying the output). Finally, the states with $\ell > c-1$ are discarded. 

Using the PTM sequence as an example, we show how to apply this approach in \texttt{Walnut} to obtain a $2$-DFAO generating $(h(n))_{n \in \N}$ in a constructive fashion. To begin, we verify that the block $(0,1,0)$ occurs in $(T_n)_{n \in \N}$ with bounded gaps. 
The following command creates a DFA accepting precisely $n=n_k > 0$, namely, such that $(T_n,T_{n+1},T_{n+2})=(0,1,0)$: \vspace{3pt} \\
\texttt{def block "T[n]=@0 \& T[n+1]=@1 \& T[n+2]=@0":} \vspace{3pt} \\
To verify that there are infinitely many such $n$, we execute the command \vspace{3pt} \\
\texttt{eval inf "Am En n>m \& \$block(n)":} \vspace{3pt} \\
which returns \texttt{TRUE}, as expected. We may now construct a DFA which accepts precisely $c \in \N_+$ equal to the gaps $n_{k+1}-n_{k}$: \vspace{3pt} \\
\texttt{def gaps "c>0 \& En (\$block(n) \& \$block(n+c) \& \\ (Aj (j<c-1)  =>  \textasciitilde \$block(n+j+1)))":} \vspace{3pt} \\
By examining the result we see that this DFA only accepts $3,5,7,9$, and thus we may take $c=9$ (knowing that $n_1 = 3$ and $n_0=0$ by our convention).

We thus use the transducer $\mathcal{T}_9$, stored in the file \texttt{Tr9.txt} where, the state labels and outputs are encoded with natural numbers rather than pairs of rational functions (so that they can be processed by \texttt{Walnut}). Then the following commands define the shift $\mathbf{\tilde{u}} = (T_{n+3})_{n\in\N}$ and the transduced sequence $\mathcal{T}(\mathbf{\tilde{u}})$. 
\vspace{3pt} \\
\texttt{def T\_shift\_DFA "T[n+3]=@1":} \\
\texttt{combine T\_shift T\_shift\_DFA:} \\
\texttt{transduce H Tr9 T\_shift:} \\
The obtained msd-first DFAO $\mathcal{H}$, stored in the file \texttt{H.txt}, generates a sequence of natural numbers which is a one-to-one coding (inherited from $\mathcal{T}_9$) of  $(h(n))_{n \in \N}$. In order to compare the result with the DFAO $\mathcal{K}$ in Figure \ref{fig_K}, we manually modify the outputs in $\mathcal{H}$ so that they encode the $h_i$ in the ``old'' way: $h_i \mapsto i$. 
The result is given in the file \texttt{H1.txt}. 
Finally, we can visually check that $\mathcal{H}_1$ is precisely the same as $\mathcal{K}$ or use the following command in \texttt{Walnut}:
\vspace{3pt} \\
\texttt{eval test "An K[n]=H1[n]":} \\
which returns $\texttt{TRUE}$, as expected. 

The same approach works for any binary sequence $\mathbf{u}$. In the more general case \eqref{fdefinition2}, if we have two sequences $\mathbf{u}$ and $\mathbf{v}$, we may encode them with a single one $\mathbf{w} = (w_n)_{n \in \N}$, where $w_n = 2u_n + v_n$. Then the conditions in Theorem \ref{main} can be given in terms of $(w_n,w_{n+1},w_{n+2})$, and the corresponding transducer needs to accept inputs from the set $\{0,1,2,3\}$. In this case the transduction can be more time consuming (especially for a large $c$), hence the guessing approach as in Theorem \ref{PTMcase} may be more effective.

\section{General results}\label{sec:general}

The property that the block $010$ occurs in the sequence $\mathbf{u}$ with bounded gaps is rare among all infinite binary sequences. 
Indeed, such sequences $\mathbf{u}$ do not contain all finite blocks as contiguous subsequences, and thus they constitute a set of Bernoulli product measure $0$. Hence, it is interesting to see what may happen to $\mathcal{V}(f)$ if we relax this assumption. In particular, we present certain general results which guarantee the infinitude of the set of quotients $\cal{V}(f)$. Moreover, in case of eventually periodic sequence ${\bf u}$, we give quite precise results concerning a kind of trichotomy in the shape of $\cal{V}(f)$. 

We start with $\mathbf{u} \in \{0,1\}^\N$ containing blocks of the form
$010^d$ for arbitrarily large $d$, which comprise almost all binary sequences. 
First, we state a simple but useful observation. 
Here and in the sequel by $(g(n))_{n\in\N}$ we mean a linear recurrence sequence defined as $g(0)=g(1)=1$ and
$$
g(n)=ag(n-1)+bg(n-2) \quad  \text{for } n\geq 2,
$$
where $a,b$ are the same as in the definition \eqref{fdefinition} of $f(n)$.

\begin{prop}\label{01000}
Assume that there exist integers $n_0\geq 2, d\geq 2$ such that $u_{n_0}=0, u_{n_0+1}=1$, and $u_{n_0+j}=0$ for each $j\in\{2,\ldots ,d\}$. 
If $f(n_0)\neq 0$, then we have the identity
$$
\frac{f(n_0+j+1)}{f(n_0+j)}=\frac{g(j+1)}{g(j)}
$$
for $j\in\{0,1,\ldots ,d-1\}$.
\end{prop}
\begin{proof}
   Since $u_{n_0}=0$ and $u_{n_0+1}=1$, we have $$f(n_0)=f(n_0+1)=af(n_0-1)+bf(n_0-2).$$ Then, by induction on $j\in\{0,1,\ldots ,d-1\}$ and the fact that $u_{n_0+j}=0$ for $i\in\{2,\ldots ,d\}$ we get that $$f(n_0+j)=f(n_0)g(j),\ j\in\{0,1,\ldots, d\},$$ which ends the proof.
\end{proof}

\begin{thm}\label{main3}
   Assume that for each $d\in\N$ there exists an integer $n_d\geq 2$ such that $u_{n_d}=0, u_{n_d+1}=1$, and $u_{n_d+j}=0$ for each $j\in\{2,\ldots ,d\}$. 
If $f(n_d)\neq 0$ for infinitely many values of $d$, then 
$$ \mathcal{V}(g) \subset \mathcal{V}(f).$$
If additionally, that $a\neq 0$, $a+b\neq 1$ and $b\not\in\left\{-a^2,-\frac{a^2}{2},-\frac{a^2}{3}\right\}$, then the set $\cal{V}(f)$ is infinite.
\end{thm}
\begin{proof}
    The inclusion $\mathcal{V}(g) \subset\cal{V}(f)$ follows directly from Proposition \ref{01000}. It suffices to show that the set $\mathcal{V}(g)$ is infinite if $a,b$ satisfy the assumptions.
    
    Let $\alpha, \beta\in\C$ be the roots of the polynomial $X^2-aX-b$. If $0\neq\alpha\neq\beta$, then there exists $\gamma, \delta\in\C$ such that $g(n)=\gamma\alpha^n+\delta\beta^n$ for each $n\in\N$. Consequently, 
    $$\frac{g(n+1)}{g(n)}=\frac{\alpha\gamma+\beta\delta\left(\frac{\beta}{\alpha}\right)^n}{\gamma+\delta\left(\frac{\beta}{\alpha}\right)^n},$$
    so the set $\mathcal{V}(g)$ is finite only if $\gamma\delta=0$ or $\frac{\beta}{\alpha}$ is a root of unity. The condition $\gamma\delta=0$ implies that $g$ is a geometric progression. Because $g(0)=g(1)=1$, we have $g(n)=1$ for each $n\in\N$ but this is impossible as $g(2)=a+b\neq 1$. Note that $\alpha,\beta$ belong to some quadratic extension of $\Q$, so does $\frac{\beta}{\alpha}$. Moreover, a root of unity has algebraic degree at most equal to $2$ exactly when its multiplicative order belongs to the set $\{1,2,3,4,6\}$. Consequently, if $\frac{\beta}{\alpha}$ is a root of unity, then $\frac{\beta}{\alpha}\in\left\{-1, \pm i, \frac{-1\pm i\sqrt{3}}{2}, \frac{1\pm i\sqrt{3}}{2}\right\}$. One can prove that
    \begin{itemize}
        \item $\frac{\beta}{\alpha}=-1$ if and only if $a=0$,
        \item $\frac{\beta}{\alpha}=\pm i$ if and only if $b=-\frac{a^2}{2}$,
        \item $\frac{\beta}{\alpha}=\frac{-1\pm i\sqrt{3}}{2}$ if and only if $b=-a^2$,
        \item $\frac{\beta}{\alpha}=\frac{1\pm i\sqrt{3}}{2}$ if and only if $b=-\frac{a^2}{3}$.
    \end{itemize}
    All of the above possibilities are excluded by the assumptions of the theorem.

    We are left with the case $\alpha=\beta$. The assumption $a\neq 0$ ensures that $\alpha\neq 0$. Then there exists $\gamma, \delta\in\C$ such that $g(n)=(\gamma n+\delta)\alpha^n$ for each $n\in\N$. Hence, 
    $$\frac{g(n+1)}{g(n)}=\alpha\frac{\gamma n+\gamma+\delta}{\gamma n+\delta},$$
    so the set $\mathcal{V}(g)$ is finite only if $\gamma=0$. This means that $g$ is a geometric progression and we have just seen that this case is excluded because $g(0)=g(1)=1\neq a+b=g(2)$.
\end{proof}

\begin{ex}\label{counterex}
Let $\mathbf{u}$ be the characteristic sequence of the set of powers of 2, i.e., $u_{n}=1$ for $n=2^{k}$ and 0 otherwise. Taking $n_{d}=2^{d}-1$ for each $d\in\N_{+}$ we see that the assumptions of Theorem \ref{main3} are satisfied. We then have the identity
$$
\frac{f(n+1)}{f(n)}=\frac{g(n-2^{\lfloor\log_{2}(n+1)\rfloor}+2)}{g(n-2^{\lfloor\log_{2}(n+1)\rfloor}+1)}
$$
for $n\geq 4$.
\end{ex}

We now prove the infinitude of $\cal{V}(f)$ under a weaker assumption on $\mathbf{u}$ than in the previous theorem, namely that arbitrarily long blocks of zeros occur. In return, we need to assume that $a^2+4b$ is not a square of an integer.

\begin{thm}\label{000}
Assume that the following conditions hold:
\begin{enumerate}
    \item for each $d\in\N$ there exists an integer such that $u_{n_d+j}=0$ for each $j\in\{0,\ldots ,d\}$;
    \item $a\neq 0$, $b\not\in\left\{-a^2,-\frac{a^2}{2},-\frac{a^2}{3}\right\}$ and $a^2+4b$ is not a square of an integer;
    \item $f(n_d-1)f(n_d-2)\neq 0$ for infinitely many $d$.
\end{enumerate}
Then the set $\cal{V}(f)$ is infinite.
\end{thm}
\begin{proof}
Let $n_d \geq 2$ satisfy (3).
Then the sequence given by the formula $g_d(j)=f(n_d-2+j)$, $j\in\{0,\ldots ,d+2\}$, is a non-zero linear binary recurrent sequence with characteristic polynomial $X^2-aX-b$. Since the determinant $a^2+4b$ of this polynomial is not a square of an integer, its roots $\alpha ,\beta$ are distinct and irrational. Thus, there exist $\gamma_d , \delta_d\in\C$ such that $g_d(j)=\gamma_d\alpha^j+\delta_d\beta^j$ for $j\in\{0,\ldots ,d+2\}$. Because $g_d(j)$ are integers, none of $\gamma_d , \delta_d$ can be equal to $0$. Then 
$$\frac{f(n_d-1+j)}{f(n_d-2+j)}=\frac{g_d(j+1)}{g_d(j)}=\frac{\alpha\gamma_d+\beta\delta_d\left(\frac{\beta}{\alpha}\right)^n}{\gamma_d+\delta_d\left(\frac{\beta}{\alpha}\right)^n}$$ for $j\in\{0,\ldots ,d+1\}$. As in the previous proof, condition (2) ensures that $\frac{\beta}{\alpha}$ is not a root of unity.
Hence, the set 
$$\left\{\frac{f(n_d-1+j)}{f(n_d-2+j)}:\ j\in\{0,\ldots ,d+1\}\right\}$$ has exactly $d+2$ distinct elements. Since we can take $d$ arbitrarily large, we conclude that the set $\cal{V}(f)$ is infinite.
\end{proof}

We now turn to the case where $\textbf{u}$ is ultimately periodic. 
We then have an interesting trichotomy in the topological type of the set $\cal{V}(f)$.

\begin{thm}\label{period}
Let $\bf{u}$ be an ultimately periodic binary sequence of period $l$ and preperiod $m$ (i.e.\ $u_{n+l}=u_n$ for $n\geq m$).
Then the set $\cal{V}(f)$ has one of the following forms:
\begin{itemize}
    \item $\cal{V}(f)$ is finite; to be more precise, the sequence $(h(n))_{n\in\N}$ is ultimately periodic of period $l$ and preperiod $t+l$ or of period $rl$ and preperiod $t$, where $r\in\{1,2,3,4,6\}$ and $$t=\begin{cases}
    \max\{m-3,0\} & \text{ if } u_n=1 \text{ for } n\geq m, \\
    \min\{n\geq \max\{m-2,0\}: u_{n+2}=0\} & \text{ otherwise};
    \end{cases}$$
    \item has at most $l$ accumulation points (with respect to natural topology on $\R\cup\{\infty\}$);
    \item is dense in $\R$.
\end{itemize}
(The periods and preperiods are not necessarily the least possible.)
\end{thm}
\begin{proof}
    Let us start with the case of $u_n=1$ for $n\geq m$. Then the sequence $(f(n))_{n\geq t}$ is linear recurrent of third order, where $t=\max\{0,m-3\}$. Hence $$f(n)=\gamma_1\alpha_1^{n-t}+\gamma_2\alpha_2^{n-t}+\gamma_3\alpha_3^{n-t}, \quad n\geq t$$ for and some $\gamma_1, \gamma_2, \gamma_3\in\C$, where $\alpha_1, \alpha_2, \alpha_3\in\C$ are pairwise distinct roots of the polynomial $X^3-aX-b$, or $$f(n)=(\gamma_1+\delta_1(n-t))\alpha_1^{n-t}+\gamma_2\alpha_2^{n-t}, \quad n\geq t$$ for and some $\gamma_1, \delta_1, \gamma_2\in\C$, where $\alpha_1\in\C$ is the double root of the polynomial $X^3-aX-b$ and $\alpha_2\in\C$ is the simple one (in this case $\alpha_1, \alpha_2, \gamma_1, \gamma_2, \delta_1\in\Q$). Note that we exclude the case of triple root of the polynomial $X^3-aX-b$ as it holds only for $a=b=0$, which emerges with ultimately zero sequence.

    Let $\alpha_1, \alpha_2, \alpha_3\in\C$ be pairwise distinct. Assume without loss of generality that $|\alpha_1|\geq |\alpha_2|\geq |\alpha_3|$. If $\gamma_1\neq 0$, then
    $$h(n) = \frac{f(n+1)}{f(n)}=\frac{\alpha_1+\frac{\gamma_2}{\gamma_1}\alpha_2\left(\frac{\alpha_2}{\alpha_1}\right)^{n-t}+\frac{\gamma_3}{\gamma_1}\alpha_3\left(\frac{\alpha_3}{\alpha_1}\right)^{n-t}}{1+\frac{\gamma_2}{\gamma_1}\left(\frac{\alpha_2}{\alpha_1}\right)^{n-t}+\frac{\gamma_3}{\gamma_1}\left(\frac{\alpha_3}{\alpha_1}\right)^{n-t}},\quad n\geq t.$$
    Hence, if $|\alpha_1|>|\alpha_2|$, then $\cal{V}(f)$ has exactly one accumulation point if $(\alpha_2\gamma_2,\alpha_3\gamma_3)\neq (0,0)$ or the sequence $(h(n))_{n\in\N}$ is ultimately constant with preperiod $t+1$ if $(\alpha_2\gamma_2,\alpha_3\gamma_3)=(0,0)$. If $|\alpha_1|=|\alpha_2|>|\alpha_3|$, $\frac{\alpha_2}{\alpha_1}$ is not a root of unity, and $\gamma_1\gamma_2\neq 0$, then $\cal{V}(f)$ has the set of accumulation points being an image of a unit circle under some M\"{o}bius transformation. On the other hand, $\cal{V}(f)\subset\R$, so then $\cal{V}(f)$ is dense in $\R$. If $|\alpha_1|=|\alpha_2|=|\alpha_3|$, then $\{\alpha_1,\alpha_2,\alpha_3\}=\{\alpha,\alpha z,\alpha z^{-1}\}$ for some $\alpha\in\R\backslash\{0\}$ and $z\in\C$ with $|z|=1$. By Viete's formula $\alpha_1+\alpha_2+\alpha_3=0$ (recall that $\alpha_1,\alpha_2,\alpha_3$ are the roots of the polynomial $X^3-aX-b$) we get $1+z+z^2=0$, so $z$ is a primitive root of unity of order $3$. This is why we now consider the more general case of $|\alpha_1|=|\alpha_2|\geq |\alpha_3|$ and $\frac{\alpha_2}{\alpha_1}$ being a root of unity. Then $\frac{\alpha_2}{\alpha_1}$ is a primitive root of unity of order $2, 3, 4$ or $6$ as it is an algebraic number of degree at most $3$ (recall that $\alpha_1\neq\alpha_2$). However, if $\frac{\alpha_2}{\alpha_1}$ is a root of unity of degree $4$ or $6$, then $\alpha_2=\bar{\alpha}_1$ and consequently
    $$|\alpha_1+\alpha_2+\alpha_3|\geq 2|\alpha_1|\cos\frac{\pi}{4}-|\alpha_3|=\sqrt{2}|\alpha_1|-|\alpha_3|>0,$$
    which stays in contradiction with the equality $\alpha_1+\alpha_2+\alpha_3=0$ by Viete's formula. Thus, $\frac{\alpha_2}{\alpha_1}\in\{-1,e^{\pm 2\pi i/3}\}$. If $\frac{\alpha_2}{\alpha_1}=-1$, then  by Viete's formulae we have $\alpha_3=0$ and $b=0$. As a result,
    $$\frac{f(n+1)}{f(n)}=\frac{\alpha_1-\frac{\gamma_2}{\gamma_1}\alpha_1\left(-1\right)^{n-t}}{1+\frac{\gamma_2}{\gamma_1}\left(-1\right)^{n-t}},\quad n\geq t+1,$$
    so the sequence $(h(n))_{n\in\N}$ is ultimately periodic with period $2$ and preperiod $t+1$. If $\frac{\alpha_2}{\alpha_1}=e^{\pm 2\pi i/3}$, then by Viete's formula we have
    $$0=|\alpha_1+\alpha_2+\alpha_3|\geq 2|\alpha_1|\cos\frac{\pi}{3}-|\alpha_3|=|\alpha_1|-|\alpha_3|\geq 0.$$
    Consequently, $|\alpha_1|=|\alpha_2|=|\alpha_3|$ and we have already shown that $\{\alpha_1,\alpha_2,\alpha_3\}=\{\alpha,\alpha e^{2\pi i/3},\alpha e^{-2\pi i/3}\}$ for some $\alpha\in\R\backslash\{0\}$. Assume without loss of generality that $\alpha_2=\alpha_1 e^{2\pi i/3}$ In particular, $a=0$. As a result,
    $$\frac{f(n+1)}{f(n)}=\alpha\frac{1+\frac{\gamma_2}{\gamma_1}\left(e^{2\pi i/3}\right)^{n-t+1}+\frac{\gamma_3}{\gamma_1}\left(e^{-2\pi i/3}\right)^{n-t+1}}{1+\frac{\gamma_2}{\gamma_1}\left(e^{2\pi i/3}\right)^{n-t}+\frac{\gamma_3}{\gamma_1}\left(e^{-2\pi i/3}\right)^{n-t}},\quad n\geq t,$$
    so the sequence $(h(n))_{n\in\N}$ is ultimately periodic with period $3$ and preperiod $t$.

     If $\gamma_1=0$ and $\gamma_2\neq 0$, then similarly as above we can conclude that:
     \begin{itemize}
         \item if $|\alpha_2|>|\alpha_3|$ and $\alpha_3\gamma_3\neq 0$, then $\cal{V}(f)$ has exactly one accumulation point,
         \item if $\alpha_3\gamma_3=0$, then the sequence $(h(n))_{n\in\N}$ is ultimately constant with preperiod $t+1$,
         \item if $|\alpha_2|=|\alpha_3|$ and $\gamma_3\neq 0$, then the sequence $(h(n))_{n\in\N}$ is ultimately periodic with period $4$ (this holds exactly when $a=-2c^2$ and $b=4c^3$ for some $c\in\Z$) or $6$ (this holds exactly when $a=-6c^2$ and $b=9c^3$ for some $c\in\Z$) and preperiod $t$ or $\cal{V}(f)$ is dense in $\R$.
     \end{itemize} 
     If $\gamma_1=\gamma_2=0$ and $\alpha_3\gamma_3\neq 0$, then we easily get that the sequence $(h(n))_{n\in\N}$ is ultimately constant with preperiod $t$.

     Now assume that $\alpha_1$ is the double root of the polynomial $X^3-aX-b$ and $\alpha_2$ is the simple one. First, let us note that $\alpha_2=-2\alpha_1$ by the Viete's formula $2\alpha_1+\alpha_2=0$. Thus,
     $$\frac{f(n+1)}{f(n)}=\alpha_2\frac{(\gamma_1+\delta_1(n-t+1))\left(-\frac{1}{2}\right)^{n-t+1}+\gamma_2}{(\gamma_1+\delta_1(n-t))\left(-\frac{1}{2}\right)^{n-t}+\gamma_2},\quad n\geq t,$$
     so the sequence $(h(n))_{n\in\N}$ is ultimately constant with preperiod $t$ if $\gamma_1\gamma_2=0$ and $\delta_1=0$ or $\cal{V}(f)$ has exactly one accumulation point otherwise.

     \bigskip

     From now on we focus on the case when $u_n=0$ for infinitely many values of $n\in\N$. Let $n_0=\min\{n\geq\max\{m,2\}: u_n=0\}=t+2$. 
     Then, we see that for each $n,k\in\N$ with $n\geq 2$ and $u_n=0$ we have
     $$f(n+k)=L_{u_{n},\ldots ,u_{n+k}}(f(n-1),f(n-2)),$$
     where $L_{-2}(x_1,x_2)=x_2$, $L_{-1}(x_1,x_2)=x_1$, $L_{u_{n}}(x_1,x_2)=L_0(x_1,x_2)=ax_1+bx_2$, and
     $$
         L_{u_{n},\ldots ,u_{n+k}}(x_1,x_2)=
         \begin{cases}
             aL_{u_{n},\ldots ,u_{n+k-1}}(x_1,x_2)+bL_{u_{n},\ldots ,u_{n+k-2}}(x_1,x_2),\,&\text{if } u_{n+k}=0\\
             aL_{u_{n},\ldots ,u_{n+k-2}}(x_1,x_2)+bL_{u_{n},\ldots ,u_{n+k-3}}(x_1,x_2),\,&\text{if } u_{n+k}=1
         \end{cases}
     $$
     for $k\geq 1$ with the convention that $L_{u_{n},\ldots ,u_{n-2}}=L_{-2}$ and $L_{u_{n},\ldots ,u_{n-1}}=L_{-1}$. Since the mappings $L_{u_{n_0},\ldots ,u_{n_0+k}}$ are linear, we have 
     $$\frac{L_{u_{n_0},\ldots ,u_{n_0+k+1}}(x_1,x_2)}{L_{u_{n_0},\ldots ,u_{n_0+k}}(x_1,x_2)}=\frac{L_{u_{n_0},\ldots ,u_{n_0+k+1}}(x_1/x_2,1)}{L_{u_{n_0},\ldots ,u_{n_0+k}}(x_1/x_2,1)}.$$
     Then, for each $k\geq -1$ we may write
     \begin{align*}
         &h(n_0+k-1)=\frac{f(n_0+k)}{f(n_0+k-1)}=\frac{L_{u_{n_0},\ldots ,u_{n_0+k}}(f(n_0-1),f(n_0-2))}{L_{u_{n_0},\ldots ,u_{n_0+k-1}}(f(n_0-1),f(n_0-2))}\\
         =&\frac{L_{u_{n_0},\ldots ,u_{n_0+k}}(h(n_0-2),1)}{L_{u_{n_0},\ldots ,u_{n_0+k-1}}(h(n_0-2),1)}=:T_{u_{n_0},\ldots ,u_{n_0+k}}(h(n_0-2)).
     \end{align*}
     In particular,
     \begin{align*} 
         h(n_0+jl+i-1)=T_{u_{n_0},\ldots ,u_{n_0+i}}\circ T_{u_{n_0},\ldots ,u_{n_0+l-1}}^j(h(n_0-2))
     \end{align*}
     for any $j\in\N$ and $i\in\{0,1,\ldots ,l-1\}$. All the maps of the form $T_{u_{n_0},\ldots ,u_{n_0+k}}$, $k\in\N$, are M\"{o}bius transformations or constant functions. Hence, the form of the sequence $(h(n))_{n\geq n_0-2}$ follows from the form of the sequence $(T_{u_{n_0},\ldots ,u_{n_0+l-1}}^j(h(n_0-2)))_{j\in\N}$. Thus, we need to explore the cardinality and topological type of the orbit of a given rational (real) number under a given M\"obius transformation with integral coefficients.

     For a given field $K$ there is an isomorphism between the group $GL_2(K)/\{\lambda\cdot\textbf{1}_2: \lambda\in K\backslash\{0\}\}$ (here $\textbf{1}_2$ denotes the $2\times 2$ identity matrix) and the group of M\"obius transformations over $K$ induced by the homomorphism
     $$\left[\begin{array}{cc}
         r & s \\
         t & u    \end{array}\right]\mapsto\left\{x\mapsto\frac{rx+s}{tx+u}\right\}.$$
    Hence, the analysis of M\"obius transormations boils down to the one of $2\times 2$ matrices.

    Since every complex-valued matrix is similar to some upper-triangular one, every complex M\"obius transformation is conjugated to an affine mapping. Indeed, let $T(x)=\frac{rx+s}{tx+u}$. Denote by $\lambda_1,\lambda_2$ the (not necessarily distinct) eigenvalues of the matrix $M=\left[\begin{array}{cc}
    r & s \\
    t & u \end{array}\right]$. If $t\neq 0$, then $M$ has an eigenvector of the form $\left[\begin{array}{c}
    z_1 \\
    1 \end{array}\right]$ for the eigenvalue $\lambda_1$ (actually, $z_1$ is a fixed point of $T$). Thus,
    $$\left[\begin{array}{cc}
    0 & 1 \\
    1 & -z_1 \end{array}\right]\left[\begin{array}{cc}
    r & s \\
    t & u \end{array}\right]\left[\begin{array}{cc}
    z_1 & 1 \\
    1 & 0 \end{array}\right]=\left[\begin{array}{cc}
    \lambda_1 & t \\
    0 & \lambda_2 \end{array}\right].$$
    Hence, the dynamical properties of $T$ are the same as the affine function given by the formula $F(x)=\frac{\lambda_1}{\lambda_2}x+\frac{t}{\lambda_2}$. It is easy to show (see \cite{B}) that if $x_0$ is not a fixed point of $F$, then:
    \begin{enumerate}
        \item $\lim_{n\to\infty}F^n(x_0)=\infty$ if $\left|\frac{\lambda_1}{\lambda_2}\right|>1$ or $\frac{\lambda_1}{\lambda_2}=1$ and $t\neq 0$,
        \item the sequence $(F^n(x_0))_{n\in\N}$ is convergent to a fixed point $w_0\neq\infty$ of $F$ if $\left|\frac{\lambda_1}{\lambda_2}\right|<1$,
        \item the set $\{F^n(x_0): n\in\N\}$ is dense in a circle or a line in a complex plane if $\left|\frac{\lambda_1}{\lambda_2}\right|=1$ and $\frac{\lambda_1}{\lambda_2}$ is not a root of unity,
        \item the set $\{F^n(x_0): n\in\N\}$ is finite if  $\frac{\lambda_1}{\lambda_2}\neq 1$ is a root of unity or $\frac{\lambda_1}{\lambda_2}=1$ and $t=0$; then $\#\{F^n(x_0): n\in\N\}$ is equal to the multiplicative order of $\frac{\lambda_1}{\lambda_2}$.
    \end{enumerate}
    The above remains true after replacing $F$ by $T$. From now on we assume that $r,s,t,u\in\R$ and $t$ may be equal to $0$. When the case $(3)$ takes place, then the closure of the set $(T^n(x_0))_{n\in\N}$ in $\C\cup\{\infty\}$ is $\R\cup\{\infty\}$ as $T(\R\cup\{\infty\})=\R\cup\{\infty\}$. Now assume further that $r,s,t,u\in\Z$ and consider the case $(4)$. Then $\lambda_1,\lambda_2$ lie in some quadratic extension of $\Q$, so does $\frac{\lambda_1}{\lambda_2}$. We already know that the multiplicative order of $\frac{\lambda_1}{\lambda_2}$ belongs to the set $\{1,2,3,4,6\}$. Note that every value from the set $\{2,3,4,6\}$ can be attained as a multiplicative order of $\frac{\lambda_1}{\lambda_2}$ for $T=T_{u_{n_0},\ldots ,u_{n_0+l-1}}$. Indeed, if $\textbf{u}$ is ultimately constant and equal to $0$, then:
    \begin{itemize}
        \item $2$ can be attained for $a=0$ and arbitrary $b\neq 0$,
        \item $3$ can be attained for arbitrary $a\neq 0$ and $b=-a^2$,
        \item $4$ can be attained for arbitrary even $a\neq 0$ and $b=-\frac{a^2}{2}$,
        \item $6$ can be attained for arbitrary $a\neq 0$ divisible by $6$ and $b=-\frac{a^2}{3}$.
    \end{itemize}
    We do not need to check if $\frac{\lambda_1}{\lambda_2}$ can be equal to $1$ for $T=T_{u_{n_0},\ldots ,u_{n_0+l-1}}$ as then $(T^n(x_0))_{n\in\N}$ is constant or convergent. We know that $(T^n(x_0))_{n\in\N}$ is convergent if e.g. $\textbf{u}$ is ultimately constant and equal to $0$, $a\neq 0$ and $a^2+4b>0$ and $(T^n(x_0))_{n\in\N}$ is constant if e.g. $\textbf{u}$ is constant and equal to $0$, $a^2+4b$ is a square of an integer and $x_0=\frac{q}{p}$ is a fixed point of $T$.

    Combining the above discussion with the fact that 
    $$\{h(n): n\geq n_0-2\}=\bigcup_{i=-1}^{l-2}\{T_{u_{n_0},\ldots ,u_{n_0+i}}\circ T_{u_{n_0},\ldots ,u_{n_0+l-1}}^j(h(n_0-2)): j\in\N\},$$
    we get the statement of the theorem. Note that if $T_{u_{n_0},\ldots ,u_{n_0+l-1}}$ is constant, then $(h(n))_{n\geq n_0+l-2}$ is periodic of period $l$ (this situation may happen if $a=0$, $b=1$ and the periodic part of the sequence $\textbf{u}$ is $``01111"$).         
\end{proof}

\begin{rem}
    Recall that the Kepler set of a given sequence $(a_n)_{n\in\N}$ is the set of all the partial limits of the sequence $\left(a_{n+1}/a_n\right)_{n\in\N}$. The notion of a Kepler set (as a generalization of Kepler limit defined by Fiorenza and Vincenzi in \cite{FV}) as $\lim_{n\to\infty}(a_{n+1}/a_n)$) was introduced by Berend and Kumar in \cite{BK}, who were interested in Kepler sets of complex linear binary recurrent sequences with constant coefficients. Theorem \ref{period} shows that if $\textbf{u}$ is ultimately periodic, then the topological type of the sequence $(h(n))_{n\in\N}$ is the same as in the case when $(f(n))_{n\in\N}$ is a linear binary recurrent sequence (compare with \cite{BHP, BK}). As one can see, with the use of Theorem \ref{period} we may deduce the form of the Kepler set of $(f(n))_{n \in \N}$.
\end{rem}

\begin{rem} \label{rem_periodic_010}
    If we assume in Theorem \ref{period} that $\textbf{u}$ contains the block $010$ in its periodic part (so $\mathcal{R}(\mathbf{u})$ has bounded gaps), then  $(h(n))_{n\in\N}$ is eventually periodic.
    However, the converse is not true, as evidenced by the example $u_{3n}=0, u_{3n+1}=u_{3n+2}=1$ and $a^2+2b=0$. In such a case $(f(n))_{n \in \N}$ consists of six interleaving geometric sequences, each with quotient $-b^2$ (and all nonzero, unless $a=b=0$ or $a=2,b=-2$).  
\end{rem}

\section{Final comments}\label{sec:comments}

In the light of our results one can ask what is going on in the case when we consider the sequence $(f(n))_{n\in\N}$ defined by \eqref{fdefinition} but with the sequence ${\bf u}=(u_{n})_{n\in\N}$ not satisfying the assumptions of the earlier results.

As already discussed in Section \ref{sec:fin_auto}, one may ask whether the condition for finiteness in Proposition \ref{main2} is also necessary.

\begin{ques}
Assume that $\mathcal{V}(f)$ is finite (when $f(n)$ are treated as polynomials in $a,b$). Is it true that $\mathcal{R}(\mathbf{u})$ has bounded gaps?
\end{ques}

We strongly believe that this is indeed the case. As Remark \ref{rem_periodic_010} shows, this may not be true in general for fixed $a,b$ though. In such a case we may extend the question to cover the first case in Theorem \ref{period}.

\begin{ques}
Assume that for some fixed $a,b$ the set $\mathcal{V}(f)$ is finite and $f(n) \neq 0$ for infinitely many $n$. Is it true that $\mathbf{u}$ is eventually periodic or $\mathcal{R}(\mathbf{u})$ has bounded gaps?
\end{ques}

We expect that the answer is again affirmative for almost all pairs $(a,b)$ in the sense of asymptotic density.

In the case where $\mathcal{V}(f)$ is infinite, we may also ask more precisely for its structure.

\begin{ex}
Let us consider $a=b=1$ and $u_{n}=T_{\lfloor n/2\rfloor}$, where $(T_{n})_{n\in\N}$ is the PTM sequence. We expect that in this case the set $\cal{V}(f)$ is infinite and its structure is difficult to describe, as can we see in Figure \ref{fig3}. 
\end{ex}

\begin{figure}
    \centering
    \includegraphics[width=1\linewidth]{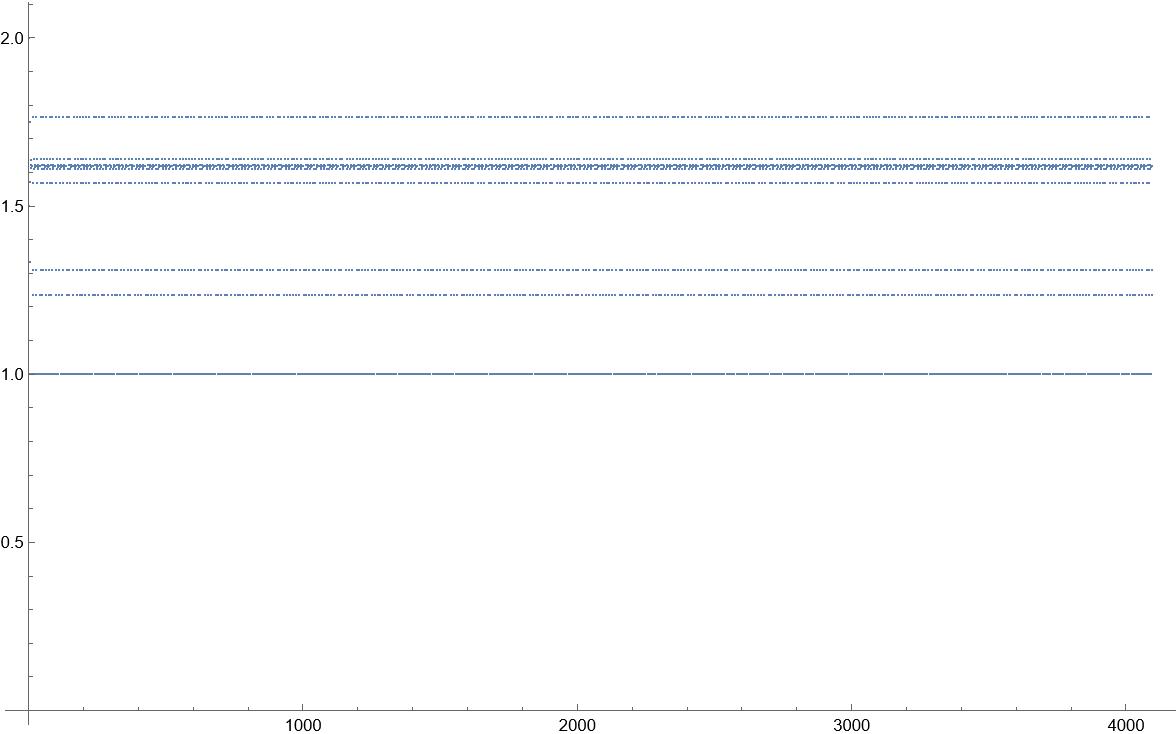}
    \caption{Plot of $f(n+1)/f(n)$ for $n\in\{1,\ldots,2^{12}\}$}.
    \label{fig3}
\end{figure}

The above example and other computer experiments suggest the following.

\begin{conj}
Assume that ${\bf u}$ is not eventually periodic, the set $\mathcal{R}(\mathbf{u})$ is finite, and the set $\{n\in\N:\;u_{n}=0,\ u_{n+1}=u_{n+2}=u_{n+3}=1\}$ is infinite. Then, for almost all pairs $(a, b)$, in the sense of asymptotic density, the set $\cal{V}(f)$, where $f$ is defined by \eqref{fdefinition}, has infinitely many accumulation points. More precisely, the set of accumulation points of the set $\cal{V}(f)$ in $\R\cup\{\infty\}$ is homeomorphic to a union of the Cantor set and a finite set or is the whole set $\R\cup\{\infty\}$. 
\end{conj}

To conclude, we note that our setting can be extended to recurrence relations of the type \eqref{fdefinition} of any order $r \geq 2$. 
For example, we may obtain a generalization of Proposition  \ref{main2}. The proof is analogous and left to the reader.

\begin{thm} \label{higher_order}
    Let $r\geq 2$, ${\bf u}^{(r)}\in\{0,1,\ldots ,r-1\}^\N$, $a_1,\ldots ,a_r\in\Z$, and let $(f^{(r)}(n))_{n\in\N}$ be defined by
    \begin{align*}
        f^{(r)}(n)&=1 \text{ for } n\leq r-1,\\
        f^{(r)}(r)&=\sum_{j=0}^r a_j,\\
        f^{(r)}(n)&=\sum_{j=1}^r a_jf^{(r)}(n-j-u^{(r)}_n) \text{ for } n\geq r+1.
    \end{align*}
If the set
$$
\{n\in\N:\ (u^{(r)}_n,\ldots ,u^{(r)}_{n+r})=(0,1,\ldots ,r-1,0)\} \cup \{0\} =\{n_{0}=0 < n_{1}< \cdots\}
$$
has gaps bounded by $c$, then $\mathcal{V}(f^{(r)})$
has at most $1+\frac{r^{c-r+1}-1}{r-1}$ elements.

Moreover, if the sequence $\mathbf{u}^{(r)}$ is additionally $k$-automatic, then the sequence $(f^{(r)}(n+1)/f^{(r)}(n))_{n \in \N}$ is also $k$-automatic.
\end{thm}

\section*{Acknowledegements}
The research of the authors was supported by a grant of the National Science Centre (NCN), Poland, no. UMO-2019/34/E/ST1/00094.

We are grateful to the referee for many useful comments and suggestions.

\end{document}